\documentclass[a4paper]{amsart}
\usepackage{amsmath,amssymb,amscd,amsthm,enumerate}
\usepackage{graphicx}

\newtheorem{Def}{Definition}
\newtheorem{Th}[Def]{Theorem}
\newtheorem{Prop}[Def]{Proposition}
\newtheorem{Lem}[Def]{Lemma}
\newtheorem{Cor}[Def]{Corollary}
\newtheorem{Rem}[Def]{Remark}
\newtheorem{Ex}[Def]{Example}

\newtheorem{Fact}[Def]{Fact}

\newcommand{\Q}{\mathbb{Q}}

\newcommand{\Z}{\mathbb{Z}}

\newcommand{\C}{\mathbb{C}}

\newcommand{\CB}{\mathcal{B}}

\newcommand{\CE}{\mathcal{E}}

\newcommand{\DS}{\displaystyle }

\newcommand{\al}{\alpha }
\newcommand{\be}{\beta }
\newcommand{\ga}{\gamma }
\newcommand{\de}{\delta }
\newcommand{\Ga}{\Gamma }
\newcommand{\De}{\Delta }

\newcommand{\vep}{\varepsilon }
\newcommand{\vph}{\varphi }
\newcommand{\la}{\lambda }
\newcommand{\om}{\omega }
\newcommand{\na}{\nabla }
\newcommand{\pa}{\partial }
\newcommand{\ot}{\otimes }

\newcommand{\Ker}{{\rm Ker}}

\newcommand{\mat}[2]{
  \left( \begin{array}{#1}
    #2
  \end{array} \right)
}
\newcommand{\TP}[1]{
  \sideset{^t}{}{\mathop{#1}}
}
\newcommand{\GA}[1]{
  \Gamma \left( #1 \right)
}

\newcommand{\diag}{{\rm diag}}
\newcommand{\tpi}{2\pi \sqrt{-1}}

\title[Contiguity relations of $F_D$]
{Contiguity relations of Lauricella's $F_D$ revisited}
\author[Y. Goto]{Yoshiaki GOTO}
\address{Department of Mathematics, 
Graduate School of Science,
Kobe University, 
Kobe 657-8501, Japan}
\email{y-goto@math.kobe-u.ac.jp}
\dedicatory{}

\date{}

\begin{document}

\maketitle

\begin{abstract}
  We study contiguity relations of Lauricella's hypergeometric 
  function $F_D$, 
  by using the twisted cohomology group and the intersection form. 
  We derive contiguity relations from those in  
  the twisted cohomology group 
  and give the coefficients in these relations 
  by the intersection numbers. 
  Furthermore, we construct twisted cycles corresponding to
  a fundamental set of solutions to the system of differential equations 
  satisfied by $F_D$, which are expressed as Laurent series. 
  We also give the contiguity relations of these solutions. 
\end{abstract}

\section{Introduction}
Lauricella's hypergeometric series $F_D$ of
$m$ variables $x_1$, $\ldots$, $x_m$ with complex parameters 
$a$, $b_1$, $\ldots$, $b_m$, $c$ is defined by 
\begin{align*}
  F_D (a,b,c;x)
  =\sum_{n_1,\ldots,n_m=0} ^{\infty } 
  \frac{(a,n_1+\cdots +n_m) (b_1,n_1) \cdots (b_m ,n_m)}
  {(c,n_1+\cdots +n_m) n_1! \cdots n_m!} x_1^{n_1} \cdots x_m^{n_m} ,  
\end{align*}
where $x=(x_1,\ldots,x_m)$, $b=(b_1,\ldots,b_m)$, 
$c \not\in \{ 0,-1,-2,\ldots \}$, 
and $(a,n)=\Gamma (a+n)/\Gamma (a)$. 
This series converges in the domain
$\left\{  x\in \C^m \left|
    \  |x_i|<1 \ (1\leq i \leq m)
  \right. \right\} $. 
It is known that 
$F_D(a,b,c;x)$ admits an Euler-type integral representation: 
\begin{align}
  \label{integral} 
  F_D(a,b,c;x) 
  = \frac{\Ga(c)}{\Ga(a)\Ga(c-a)} 
  \int _1^{\infty} 
  t^{\sum_{i=1}^m b_i -c} (t-1)^{c-a-1} \prod_{i=1}^m (t-x_i)^{-b_i} dt .
\end{align}

The contiguity relations of Lauricella's $F_D$ 
have been studied from several points of view. 
In the 1970s, W. Miller Jr.~\cite{Miller} gave the contiguity relations 
of $F_D$ as a representation of a Lie algebra, and 
Aomoto \cite{Aomoto} studied the contiguity relations of $F_D$ 
and its generalization to the hypergeometric functions of type $(k,n)$. 
In 1991, Sasaki \cite{Sasaki} studied the contiguity relations 
in the framework of the Aomoto-Gel'fand system 
on the Grassmannian manifold.
In 1989, an algorithmic method that used Gr\"obner bases to derive the contiguity relations was given 
by Takayama \cite{Takayama}. 
Recently, Ogawa, Takemura, and Takayama \cite{OTT} have
illustrated that the Pfaffian system and 
the contiguity relations for $F_D$ combine to give a method to evaluate 
the normalizing constant of 
the hypergeometric distribution on 
the $2$ by $N$ contingency tables 
with given marginal sums. 
On the other hand, 
Matsumoto \cite{M-FD} recently proposed a method 
that utilizes the intersection numbers of twisted cohomology groups to derive Pfaffian systems. 
In this paper, we reconsider the problem of the contiguity relations of $F_D$, 
in order to produce formulas for application to statistics \cite{OTT}. 
Matsumoto's method can be applied to derive the contiguity relations for our 
purpose, and further generalizations will be possible.

We derive the contiguity relations of 
$F_D$ by considering 
the twisted cohomology groups associated with 
the integral representation (\ref{integral}). 
We regard the contiguity relations as those between the twisted cocycles. 
To obtain the coefficients in the contiguity relations, 
we use the intersection form of the twisted cohomology group. 
In the way, we are able to derive the contiguity relations 
for the basis given in \cite{M-FD}, 
which was also used in \cite{OTT}. 
An advantage of our method is that 
it makes it easy to systematically derive the contiguity relations for a given basis of the twisted cohomology group. 

This paper is arranged as follows. 
In Sections \ref{sec-cohomology}, \ref{sec-contiguity}, and \ref{sec-proof}, 
we introduce our method for using the intersection form to
derive the contiguity relations. 
By evaluating the intersection numbers, we obtain explicit forms for the 
contiguity relations. 
In Section \ref{sec-DE}, 
we introduce the system $E_D (a,b,c)$ of differential equations 
satisfied by $F_D(a,b,c;x)$, and we introduce 
the Laurent series solution $f^{(k)}(a,b,c;x)$ to $E_D (a,b,c)$ 
and construct a fundamental set of solutions. 
In Section \ref{sec-cycle}, 
we construct the twisted cycle $r_k$ 
corresponding to the solution $f^{(k)}(a,b,c;x)$. 
Since our contiguity relations are obtained from 
those in the twisted cohomology group, 
the integration on $r_k$ gives the contiguity relations of $f^{(k)}$. 
In Section \ref{sec-application}, we present 
an application of our formula, in which we evaluate 
the normalizing constant of 
the hypergeometric distribution on 
the $2$ by $(m+1)$ contingency tables; 
this is also explained in \cite{OTT} in the context of statistics. 
We also explain how to apply our results 
when the parameters $(a,b,c)$ are integers. 

Although the contiguity relations of $F_D$ have been 
studied by several authors, 
those of the other solutions $f^{(k)}$ that 
appear in applications to statistics 
have not been studied.

\section{Twisted cohomology group and intersection pairing}
\label{sec-cohomology}
We summarize some results in \cite{AK}, \cite{CM}, and \cite{M-FD} 
that will be used in this paper. 
We consider 
the twisted cohomology group for 
$$
T_x:=\C -\{ x_0, x_1 ,\ldots ,x_m,x_{m+1} \}
$$
and the multivalued function 
\begin{align*}
  u_x(t) := \prod_{i=0}^{m+1} (t-x_i)^{\al_i} ,
\end{align*}
where 
\begin{align}
  & \nonumber
  x_0:=0,\quad x_{m+1}:=1,\\
  & \label{alpha}
  \al_0:=-c+\sum_{j=1}^m b_j ,\quad 
  \al_k:=-b_k \ (1\leq k \leq m),\quad 
  \al_{m+1}:=c-a,\quad \al_{m+2}:=a. 
\end{align}
Except in Section \ref{sec-application}, 
we assume the condition 
\begin{align}
  \label{generic}
  \al_k \not\in \Z \quad (0 \leq k \leq m+2).
\end{align}

We denote the vector space consisting of the smooth $k$-forms on $T_x$ 
and that with compact support by $\CE^k(T_x)$ and $\CE^k_c(T_x)$,
respectively. 
We set $\om :=d \log u_x$ and $\na_{\om} :=d+\om \wedge$, 
where $d$ is the exterior derivative 
with respect to the variable $t$ 
(note that this is not with respect to $x_1,\ldots ,x_m$, 
which are regarded as parameters). 
The 
twisted cohomology group and that with compact support 
are defined as  
\begin{align*}
  & H^1 (T_x,\na_{\om})
  =\Ker (\na_{\om}: \CE^1(T_x) \to \CE^2(T_x)) 
  / \na_{\om} (\CE^0(T_x)), \\
  & H_c ^1 (T_x,\na_{\om})
  =\Ker (\na_{\om}: \CE^1_c(T_x) \to \CE^2_c(T_x)) 
  / \na_{\om} (\CE^0_c(T_x))  ,
\end{align*}
respectively. 
The expression (\ref{integral}) means that the integral 
$$
\int_1^{\infty} u_x \vph_0 ,\ \ 
\vph_0  :=\frac{dt}{t-1} 
$$
represents $F_D(a,b,c;x)$ modulo Gamma factors. 
By \cite{AK}, 
$H^1 (T_x,\na_{\om})$ has $(m+1)$dimensions, 
and there is a canonical isomorphism 
$\jmath : H^1 (T_x,\na_{\om}) \to H_c ^1 (T_x,\na_{\om})$; 
see also \cite[Fact 6.1]{M-FD}. 
Hereafter, we identify $H_c ^1 (T_x,\na_{\om})$ with $H^1 (T_x,\na_{\om})$. 

The intersection form $I_c$ on the twisted cohomology groups is 
the pairing between $H^1 (T_x,\na_{\om})$ and $H^1 (T_x,\na_{-\om})$, and 
it is defined as follows: 
$$
I_c (\psi ,\psi'):=\int_{T_x} \jmath (\psi) \wedge \psi' ,\quad 
\psi \in H^1(T_x,\na_{\om}) ,\ \psi' \in H^1 (T_x,\na_{-\om}). 
$$

We put 
\begin{align*}
  & \vph_{i,m+2}:=\frac{dt}{t-x_i},\quad 
  \vph_{i,j}:=\vph_{i,m+2}-\vph_{j,m+2}=\frac{(x_i-x_j)dt}{(t-x_i)(t-x_j)}, \\
  &\vph_0 =\vph_{m+1,m+2}=\frac{dt}{t-1},\quad 
  \vph_k :=\vph_{m+1,k}=\frac{(1-x_k)dt}{(t-x_k)(t-1)} , 
\end{align*}
where $0 \leq i,j \leq m+1$ and $1 \leq k \leq m$. 
The intersection numbers among these 1-forms are evaluated 
in \cite{CM}; see also \cite[Fact 6.2]{M-FD}. 
\begin{Fact}[\cite{CM}]
  \label{intersection}
  We have 
  \begin{align*}
    I_c (\vph_{i,j} ,\vph_{p,q})
    =\tpi \left(  
      \frac{\de_{i,p}-\de_{i,q}}{\al_i}-
      \frac{\de_{j,p}-\de_{j,q}}{\al_j}
    \right) ,
  \end{align*}
  where $i,j,p,q \in \{ 0,1,\ldots,m+2 \}$, and 
  $\de_{i,p}$ is the Kronecker delta. 
  Thus, the intersection matrix 
  $C(a,b,c):=\left( I_c(\vph_i ,\vph_j) \right)_{i,j=0,\ldots m}$ is 
  $$
  C(a,b,c)=2\pi \sqrt{-1} \left\{
  \frac{1}{\al_{m+1}}N+\diag \left( 
    \frac{1}{\al_{m+2}},\frac{1}{\al_1} ,\dots, \frac{1}{\al_m} \right) \right\}, 
  $$
  where 
  $$
  N=\mat{ccc}{1&\cdots&1 \\ \vdots&\ddots&\vdots \\ 1&\cdots&1}. 
  $$
  Under assumption (\ref{generic}), we have 
  $$
  \det (C(a,b,c))=(\tpi)^{m+1} \frac{-\al_0}{\prod_{i=1}^{m+2} \al_i} 
  \neq 0, 
  $$
  and hence $\vph_0 ,\ldots ,\vph_m$ form a basis of $H^1(T_x,\na_{\om})$. 
\end{Fact}

\section{Contiguity relations}
\label{sec-contiguity}
In this section, we derive the contiguity relations by using the intersection form. 

We define two column vectors of size $m+1$: 
\begin{align*}
  &F(a,b,c;x):=
  \TP{ \left(F_D(a,b,c;x) ,\ 
    \frac{x_1-1}{\al_1} \frac{\pa}{\pa x_1}F_D(a,b,c;x) ,\  \cdots ,\  
    \frac{x_m-1}{\al_m} \frac{\pa}{\pa x_m}F_D(a,b,c;x) \right) },\\
  &\tilde{F}(a,b,c;x):=
  \frac{\Ga(a) \Ga(c-a)}{\Ga(c)}F(a,b,c;x).
\end{align*}
For $(v_0,\ldots ,v_m)\in \C^{m+1}$, we regard 
$v_i$ as the $i$-th entry. 
For example, the $0$-th entry of $F(a,b,c;x)$ is $F_D(a,b,c;x)$. 
By \cite[Corollary 7.2]{M-FD}, we have 
$$
\TP{ \left( \int_1^{\infty} u_x \vph_0 ,\ldots ,\int_1^{\infty} u_x \vph_m \right) }
=\tilde{F}(a,b,c;x). 
$$
Our main theorem states the contiguity relations of 
the vector-valued function $F(a,b,c;x)$. 

For $0\leq k \leq m+1$, there exist $p_{il}^{(k)}(a,b,c;x)$'s such that  
\begin{align}
  \label{wa}
  (t-x_k)\cdot \vph_i 
  =\sum_{l=0}^m p_{il}^{(k)}(a,b,c;x) \cdot \vph_l   
\end{align}
as elements in the twisted cohomology group $H^1(T_x,\na_{\om})$. 
We put $P_k(a,b,c,x):=\left( p_{ij}^{(k)}(a,b,c;x) \right)_{i,j}$. 
\begin{Lem}\label{contiguity-lem}
  \begin{align*}
    \tilde{F}(a-1,b,c;x)&=P_{m+1}(a,b,c;x)\tilde{F}(a,b,c;x), \\
    \tilde{F}(a-1,b,c-1;x)&=P_0(a,b,c;x)\tilde{F}(a,b,c;x), \\
    \tilde{F}(a-1,b-e_k,c-1;x)&=P_k(a,b,c;x)\tilde{F}(a,b,c;x) \quad 
    (1\leq k \leq m). 
  \end{align*}
  Here, $e_k$ is the $k$-th unit vector in $\C^m$. 
\end{Lem}
For example, we have $b-e_1 =(b_1-1,b_2 ,\dots,b_m)$.
\begin{proof}
  Recall that $x_{m+1}=1$. 
  We consider the integration of (\ref{wa}) on $(1,\infty)$. 
  By (\ref{integral}), we have 
  \begin{align*}
    \int_1^{\infty} u_x \cdot (t-1)\vph_0 
    &=\int _1^{\infty} 
    t^{\sum_i b_i -c} (t-1)^{c-a} \prod_{i=1}^m (t-x_i)^{-b_i} dt \\
    &=\frac{\Ga(a-1) \Ga(c-a+1)}{\Ga(c)} F_D(a-1,b,c;x),
  \end{align*}
  which is the $0$-th entry of $\tilde{F}(a-1,b,c;x)$. 
  Then, the first equality follows. 
  The other ones are shown in an analogous way. 
\end{proof}
The following lemma is obvious. 
\begin{Lem}\label{contiguity-cor}
  \begin{align*}
    \tilde{F}(a-1,b,c;x)&=P_{m+1}(a,b,c;x) \tilde{F}(a,b,c;x), \\
    \tilde{F}(a,b,c-1;x)&=P_0(a+1,b,c;x)P_{m+1}(a+1,b,c;x)^{-1}\tilde{F}(a,b,c;x), \\
    \tilde{F}(a,b-e_k,c;x)&=P_k(a+1,b,c+1;x)P_0(a+1,b,c+1;x)^{-1}\tilde{F}(a,b,c;x).
  \end{align*}
\end{Lem}
To evaluate $P_k (a,b,c;x)$, we use the intersection form $I_c$. 
By (\ref{wa}), we have 
$$
I_c \left( (t-x_k)\vph_i ,\vph_j \right)
=\sum_{l=0}^m p_{il}^{(k)}(a,b,c;x)\cdot I_c (\vph_l ,\vph_j). 
$$
Let $Q_k(a,b,c;x):=\left( I_c \left((t-x_k)\vph_i ,\vph_j \right) \right)_{i,j}$. 
Then we obtain $Q_k(a,b,c;x) =P_k(a,b,c;x) C(a,b,c)$, that is, 
$$
P_k(a,b,c;x)=Q_k(a,b,c;x) C(a,b,c)^{-1}. 
$$
We can reduce Lemma \ref{contiguity-cor} to the relations between the $F(a,b,c;x)$'s
by using the formula $\Ga (s+1)=s\cdot \Ga (s)$.  
\begin{Th}[Contiguity relations]
  \label{main}
  We have 
  \begin{align*}
    F(a-1,b,c;x)&=D_a (a,b,c;x)F(a,b,c;x), \\
    F(a,b,c-1;x)&=D_c (a,b,c;x)F(a,b,c;x), \\
    F(a,b-e_k,c;x)&=D_k (a,b,c;x)F(a,b,c;x) \quad  (1\leq k \leq m), 
  \end{align*}
  where 
  \begin{align*}
    D_a (a,b,c;x)&:=\frac{a-1}{c-a} \cdot Q_{m+1}(a,b,c;x) \cdot C(a,b,c)^{-1}, \\
    D_c (a,b,c;x)&:=\frac{c-a-1}{c-1} \cdot Q_0(a+1,b,c;x) \cdot Q_{m+1}(a+1,b,c;x)^{-1}, \\
    D_k (a,b,c;x)&:=Q_k(a+1,b,c+1;x) \cdot Q_0(a+1,b,c+1;x)^{-1}. 
  \end{align*}
\end{Th}
Therefore, we need to evaluate the matrices $Q_k(a,b,c;x)$ explicitly. 
\begin{Prop}\label{Q}
  We have 
  \begin{align*}
    Q_k(a,b,c;x)& \\ =\tpi \Biggl\{
    &\frac{1-x_k}{\al_{m+1}} N 
    +\frac{1}{\al_{m+2}} \diag\left(0,1-x_1,\dots ,1-x_m \right)
    \cdot N \cdot \diag\left(1,0,\dots ,0 \right)   \\
    &-\frac{1}{1-\al_{m+2}} \diag\left(1,0,\dots ,0 \right) \cdot N 
    \cdot \diag\left(0, 1-x_1,\dots ,1-x_m \right)   \\
    &+\diag\left(\frac{1-x_k}{\al_{m+2}}-
      \frac{1}{1-\al_{m+2}}\left( \frac{\sum_{p=1}^{m+1} \al_p x_p}{\al_{m+2}}+1 \right), 
      \frac{x_1-x_k}{\al_1},\dots,\frac{x_m-x_k}{\al_m} \right) 
    \Biggr\} . 
  \end{align*}
\end{Prop}
We will give the proof in the next section. 
Note that 
\begin{align*}
  \diag(p_0,\dots,p_m) \cdot N\cdot \diag(q_0,\dots,q_m)
  =\mat{cccc}{p_0 q_0&p_0 q_1&\cdots&p_0q_m \\ 
    p_1 q_0&p_1 q_1&\cdots&p_1 q_m \\
    \vdots&\vdots&\vdots&\vdots \\
    p_m q_0&p_m q_1&\cdots&p_m q_m }. 
\end{align*}
\begin{Rem}\label{detQ}
  The determinant of $Q_k (a,b,c;x)$ is as follows: 
  \begin{align*}
    \det Q_k (a,b,c;x)
    =(2 \pi \sqrt{-1})^{m+1} 
    \cdot \frac{\al_0 (1+\de_{k,0}\al_0)}
    {\DS \prod_{\substack{j=1 \\ j\neq k} }^{m+2} \al_j \cdot (\al_{m+2} -1)}
    \cdot \prod_{\substack{j=0 \\ j\neq k}}^{m+1} (x_j-x_k) .
  \end{align*}
\end{Rem}
\begin{Ex}
  If $m=2$, the matrices $C(a,b,c)$ and $Q_k (a,b,c;x)$ are as follows: 
  \begin{align*}
    &C(a,b,c;x)
    =\tpi \left\{  
      \frac{1}{\al_3}\mat{ccc}{1&1&1 \\ 1&1&1 \\ 1&1&1}
      +\mat{ccc}
      {\frac{1}{\al_4}&0&0 \\ 
        0&\frac{1}{\al_1}&0 \\ 
        0&0&\frac{1}{\al_2}}
    \right\} ,\\
    &Q_k(a,b,c;x) \\
    &=\tpi \left\{  
      \frac{1-x_k}{\al_3}\mat{ccc}{1&1&1 \\ 1&1&1 \\ 1&1&1}
      +\mat{ccc}
      {\frac{1-x_k}{\al_4}
        -\frac{1}{1-\al_4} \cdot \frac{\al_1 x_1+\al_2 x_2 +\al_3+\al_4}{\al_4}&
        -\frac{1-x_1}{1-\al_4}&-\frac{1-x_2}{1-\al_4} \\ 
        \frac{1-x_1}{\al_4}&\frac{x_1-x_k}{\al_1}&0 \\ 
        \frac{1-x_2}{\al_4}&0&\frac{x_2-x_k}{\al_2}}
    \right\} .
  \end{align*}
  The first equality in Theorem \ref{main} is written as 
  \begin{align*}
    &\mat{r}{F_D (a-1,b_1,b_2,c;x_1,x_2) \\
      \frac{1-x_1}{b_1}\cdot \frac{\pa}{\pa x_1} F_D (a-1,b_1,b_2,c;x_1,x_2) \\
      \frac{1-x_2}{b_2}\cdot \frac{\pa}{\pa x_2} F_D (a-1,b_1,b_2,c;x_1,x_2)} \\
    &=\mat{ccc}{\frac{-b_1 x_1 -b_2 x_2+c-a}{c-a}
      &\frac{b_1 x_1}{c-a}&\frac{b_2 x_2}{c-a} \\ 
      \frac{(a-1)(1-x_1)}{c-a}&\frac{(a-1)(x_1-1)}{c-a}&0 \\ 
      \frac{(a-1)(1-x_2)}{c-a}&0&\frac{(a-1)(x_2-1)}{c-a}}
    \mat{r}{F_D (a,b_1,b_2,c;x_1,x_2) \\
      \frac{1-x_1}{b_1}\cdot \frac{\pa}{\pa x_1} F_D (a,b_1,b_2,c;x_1,x_2) \\
      \frac{1-x_2}{b_2}\cdot \frac{\pa}{\pa x_2} F_D (a,b_1,b_2,c;x_1,x_2)} .
  \end{align*}
  The $3\times 3$ matrix on the right-hand side is equal to 
  $D_a (a,b,c;x)=\frac{a-1}{c-a} \cdot Q_3(a,b,c;x) \cdot C(a,b,c)^{-1}$. 
\end{Ex}

\section{Proof of Proposition \ref{Q}}
\label{sec-proof}
In this section, we evaluate the intersection numbers 
that are the entries of $Q_k (a,b,c;x)$, 
by using Fact \ref{intersection}. 

We denote $\vph \sim \psi$, if $\vph$ is 
$\na_{\om}$-cohomologous to $\psi$, 
that is, 
\begin{align*}
  \vph \sim \psi &\Longleftrightarrow 
  \vph = \psi + \na_{\om} f \quad 
  {\rm for \ some} \  f \in \CE^0(T_x) , \\ 
  &\Longleftrightarrow 
  \vph \ {\rm and} \ \psi \ 
  {\rm give \ the \ same \ element \ in}\ H^1(T_x,\na_{\om}). 
\end{align*}
\begin{Lem}
  $$
  dt \sim -\frac{1}{1-\al_{m+2}}\sum_{p=1}^{m+1} 
  \al_p x_p \vph_{p,m+2}.
  $$
\end{Lem}
\begin{proof}
  This lemma follows from 
  \begin{align*}
    0 \sim \na_{\om} (t)
    &=dt+\sum_{p=0}^{m+1} \al_p \frac{t}{t-x_p}dt
    =dt+\sum_{p=0}^{m+1} \al_p \frac{t-x_p+x_p}{t-x_p}dt \\
    &=(1+\sum_{p=0}^{m+1} \al_p)dt + \sum_{p=0}^{m+1} \al_p x_p \frac{dt}{t-x_p}
    =(1-\al_{m+2})dt + \sum_{p=1}^{m+1} \al_p x_p \vph_{p,m+2}.
  \end{align*}
  Here, we use $x_0=0$ and $\sum_{p=0}^{m+2}\al_p=0$.
\end{proof}
Then, we have
\begin{align*}
  (t-x_k)\cdot \vph_{l,m+2} &=\frac{t-x_k}{t-x_l}dt
  =\frac{t-x_l+x_l-x_k}{t-x_l}dt \\
  &\sim (x_l-x_k)\vph_{l,m+2} -\frac{1}{1-\al_{m+2}}\sum_{p=0}^{m+1} 
  \al_p x_p \vph_{p,m+2}.
\end{align*}
Fact \ref{intersection} and a
straightforward calculation show the following lemma. 
\begin{Lem}
  \label{lem-intersection}
  \begin{align*}
    &I_c \left( (t-x_k)\vph_{l,m+2},\vph_j \right) \\ &= \left\{
        \begin{array}{l}
          \tpi \left( (x_l-x_k)
            \left( \frac{\de_{l,m+1}}{\al_{m+1}}+\frac{1}{\al_{m+2}} \right) 
            -\frac{1}{1-\al_{m+2}} 
            \left( \frac{\sum_{p=1}^{m+1} \al_p x_p}{\al_{m+2}}+1 \right) 
          \right) \quad (j=0), \\
          \tpi \left( (x_l-x_k)\frac{\de_{l,m+1}-\de_{l,j}}{\al_l}
            -\frac{1-x_j}{1-\al_{m+2}}\right) \quad (1\leq j \leq m).
        \end{array}
      \right.
  \end{align*}
\end{Lem}

\begin{proof}[Proof of Proposition \ref{Q}]
  Let $Q_k(i,j)$ be the $(i,j)$ entry of 
  $Q_k (a,b,c;x)$, that is, 
  $Q_k(i,j)=I_c \left( (t-x_k)\cdot \vph_i,\vph_j \right)$. 
  For $1 \leq i,j \leq m$, we have 
    \begin{align*}
      Q_k(0,0)&=I_c \left( (t-x_k)\cdot \vph_{m+1,m+2},\vph_0 \right) \\
      &=\tpi \left(  
        \frac{1-x_k}{\al_{m+1}}+\frac{1-x_k}{\al_{m+2}}
        -\frac{1}{1-\al_{m+2}} \left( \frac{\sum_{p=1}^{m+1} \al_p x_p}{\al_{m+2}}+1 \right) 
      \right), \\
      Q_k(0,j)&=I_c \left( (t-x_k)\cdot \vph_{m+1,m+2},\vph_j \right) \\
      &=\tpi \left(  
        \frac{1-x_k}{\al_{m+1}}-\frac{1-x_j}{1-\al_{m+2}} 
      \right), \\
      Q_k(i,0)&=I_c \left( (t-x_k)\cdot \vph_{m+1,m+2},\vph_0 \right) 
      -I_c \left( (t-x_k)\cdot \vph_{i,m+2},\vph_0 \right) \\
      &= \tpi \left( \frac{1-x_k}{\al_{m+1}}+\frac{1-x_i}{\al_{m+2}} \right), \\
      Q_k(i,j)&=I_c \left( (t-x_k)\cdot \vph_{m+1,m+2},\vph_j \right)
      -I_c \left( (t-x_k)\cdot \vph_{i,m+2},\vph_j \right)\\
      &=\tpi \left( \frac{1-x_k}{\al_{m+1}}+\frac{x_i-x_k}{\al_i}\de_{i,j} \right) ,
    \end{align*}
  by Lemma \ref{lem-intersection}. 
  These equalities imply Proposition \ref{Q}.   
\end{proof}

\section{Differential equations and solutions}
\label{sec-DE}
Lauricella's $F_D (a,b,c;x)$ satisfies the differential equations
\begin{align*}
  & \left[ \theta_i ( \theta +c-1 ) -x_i (\theta +a) (\theta_i +b_i) \right] f(x)=0
  \quad (1 \leq i \leq m),\\
  & \left[ (x_i -x_j) \pa_i \pa_j -b_j \pa_i +b_i \pa_j \right] f(x)=0 
  \quad (1\leq i < j \leq m),
\end{align*}
where 
$\pa_i :=\frac{\pa}{\pa x_i}$, 
$\theta_i := x_i \pa_i$, and 
$\theta := \sum_{j=1}^m \theta_j$. 
The system generated by them is called Lauricella's 
hypergeometric system $E_D (a,b,c)$ of differential equations. 
It is known that 
the $A$-hypergeometric system associated with the matrix 
$A(\De_1 \times \De_m)$ can be transformed into 
the system $E_D(a,b,c)$, 
and combinatorial methods for constructing a fundamental set 
of solutions to the $A$-hypergeometric system are known \cite{GKZ}, \cite{SST}. 
Thus, we can 
use the general method for constructing series solutions
to $A$-hypergeometric systems to obtain a fundamental set of 
solutions to $E_D (a,b,c)$ 
with generic parameters $(a,b,c)$.

\begin{Fact}[{\cite[Section 3.3]{GKZ}}, {\cite[Section 1.5]{SST}}]
  For $1\leq k \leq m$, we put 
  \begin{align*}
    f^{(k)}(a,b,c;x):=&
    \prod_{l=1}^{k-1}x_l^{-b_l} 
    \cdot x_k^{\sum_{l=1}^{k-1}b_l -c+1} \\
    &\cdot \sum_{n_1 ,\ldots ,n_m=0}^{\infty}
    \frac{1}{\Ga_{n_1,\ldots ,n_m}^{(k)}(a,b,c)} 
    \cdot \prod_{l=1}^{k-1} \left( \frac{x_k}{x_l} \right)^{n_l}
    \cdot x_k^{n_k}
    \cdot \prod_{l=k+1}^m \left( \frac{x_l}{x_k} \right)^{n_l},
  \end{align*}
  where 
  \begin{align*}
    &\Ga_{n_1,\ldots ,n_m}^{(k)}(a,b,c) \\
    &:= \GA{c-a-n_k} \cdot \prod_{\substack{1 \leq l \leq m \\ l\neq k}}\GA{1-b_l-n_l} 
    \cdot \prod_{l=1}^m \GA{1+n_l}\\
    &\cdot \GA{-\sum_{l=1}^{k} b_l+c-\sum_{l=1}^{k}n_l+\sum_{l=k+1}^m n_l}
    \cdot \GA{2+\sum_{l=1}^{k-1} b_l-c+\sum_{l=1}^{k}n_l-\sum_{l=k+1}^m n_l} .
  \end{align*}
  Then, each $f^{(k)}(a,b,c;x)$ is a solution to $E_D(a,b,c)$. 
  Moreover, the set of $F_D (a,b,c;x)$ and $f^{(k)}(a,b,c;x)$ $(1\leq k \leq m)$ 
  is a set of fundamental solutions to $E_D (a,b,c)$. 
\end{Fact}

\section{Twisted cycles corresponding to solutions}
\label{sec-cycle}
We consider the twisted homology group $H_1 (T_x, u_x)$ on $T_x$ 
that is associated with the multivalued function $u_x(t)$. 
For the definition of the twisted homology groups, 
refer to \cite{AK} and \cite{M-FD}. 
By \cite{AK}, $H_1 (T_x,u_x)$ has $(m+1)$ dimensions. 
If $(a,b,c;x)$ are generic, then 
the local solution space $Sol_x$ of $E_D (a,b,c)$ 
around $x$ can be identified with 
the twisted homology group $H_1 (T_x ,u_x)$ 
by the integration of $u_x \vph_0$; see \cite[Proposition 4.1]{M-FD}. 
Thus, there exists a twisted cycle that corresponds to 
the series solution $f^{(k)}(a,b,c;x)$. 
In this section, we construct such a cycle explicitly.

Let $\vep$ and $\xi$ be real numbers satisfying 
$$
0< \vep <\frac{1}{2}, \quad 
\xi < \min \left\{ \vep ,\frac{1}{1+\vep} \right\} .
$$
We construct the twisted cycle $r_k$ in $T_x$ with 
$x$ belonging to a small neighborhood of 
$$
x^{(k)}:=
(\xi ,\xi^2 ,\ldots ,\xi^{k-1} ,e^{-\pi \sqrt{-1}} \xi^k ,\xi^{k+1} ,\ldots ,\xi^m) .
$$
Once we construct the twisted cycle in $T_{x^{(k)}}$, 
this cycle is uniquely continued to the twisted cycle in each $T_x$. 
Thus, we may assume $x=x^{(k)}$. 
We put 
\begin{align*}
  S_x :=&\C -\left\{  
    \frac{x_k}{x_m}, \ldots ,\frac{x_k}{x_{k+1}},
    \frac{x_k}{x_{k-1}} , \ldots , \frac{x_k}{x_1} , x_k ,0,1
  \right\} , \\
  v_x (s):=&\prod_{l=1}^{k-1} \left( s-\frac{x_k}{x_l} \right)^{\al_l}
  \cdot \left( s-x_k  \right) ^{\al_{m+1}-1} 
  \cdot \prod_{l=k+1}^{m} \left( 1-\frac{x_l}{x_k} s \right)^{\al_l}  \\
  &\cdot s^{\al_{m+2}} \cdot (1-s)^{\al_k+1} \\
  =&\prod_{l=1}^{k-1} \left( 1-\frac{x_k}{x_l} \frac{1}{s} \right)^{\al_l}
  \cdot \left( 1-x_k \frac{1}{s} \right) ^{\al_{m+1}-1} 
  \cdot \prod_{l=k+1}^{m} \left( 1-\frac{x_l}{x_k} s \right)^{\al_l} \\
  &\cdot s^{\sum_{l=1}^{k-1}\al_l +\al_{m+1}+\al_{m+2}-1} \cdot (1-s)^{\al_k+1}.
\end{align*}
The last equality holds when $0 < s <1$. 
We define the twisted cycle $\tilde{r}_k$ that gives an element in $H_1 (S_x ,v_x)$. 
We put $\la_j :=e^{\tpi \al_j}$ and 
\begin{align*}
  \tilde{r}_k:=
  \frac{1}{\prod_{l=1}^{k-1}\la_l \cdot \la_{m+1} \la_{m+2} -1}C_0 \ot v_x
  +[\vep ,1-\vep] \ot v_x
  -\frac{1}{\la_k -1}C_1 \ot v_x . 
\end{align*}
Here, $C_0$ (resp. $C_1$) is the circle of center $0$ (resp. $1$) and radius $\vep$ 
with starting point $\vep$ (resp. $1-\vep$), which turns in the counterclockwise direction, 
and the branch of $v_x$ is obtained by the analytic continuation along $C_0$ (resp. $C_1$). 
Let us verify that $\tilde{r}_k$ is a twisted cycle. 
Let $D_i$ be the disk whose boundary is $C_i$ $(i=0,\ 1)$. 
Since 
$$
\left| \frac{x_k}{x_{k-1}} \right| =\xi <\vep <1 
<\frac{1}{\xi}=\left| \frac{x_k}{x_{k+1}} \right| , 
$$
we have 
\begin{align*}
  D_0 \cap (\C -S_x) =\left\{ \frac{x_k}{x_{k-1}} , \ldots , \frac{x_k}{x_1} , x_k ,0 \right\} ,\quad 
  D_1 \cap (\C -S_x) =\left\{ 1 \right\} ;
\end{align*}
see Figure \ref{fig-cycle}. 
Then, the difference between the branches of $v_x$ 
at the ending and starting points of the circle $C_0$ (resp. $C_1$) is 
$\prod_{l=1}^{k-1}\la_l \cdot \la_{m+1} \la_{m+2}$ 
(resp. $\la_k$), 
which implies that $\tilde{r}_k$ is a twisted cycle
(cf. \cite[Example 2.1]{AK}). 
\begin{figure}[h]
  \centering{
    \includegraphics[scale=0.9]{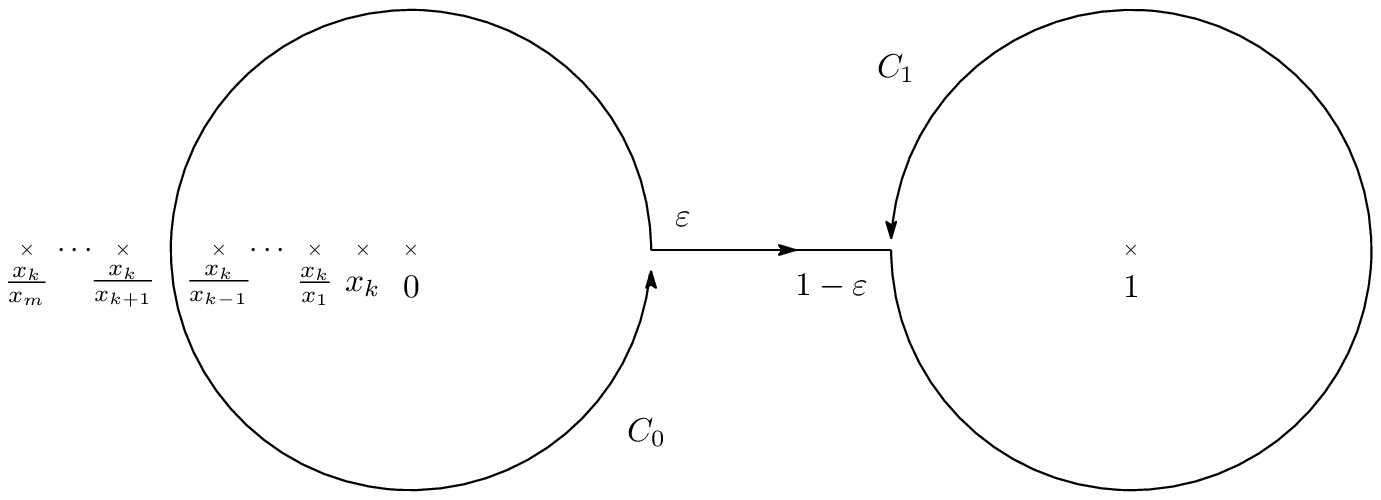} 
  }
  \caption{$\tilde{r}_k$}
  \label{fig-cycle}
\end{figure}

\begin{Lem}\label{homology-lem}
  \begin{align}
    \label{homology-lem-1}
    \int_{\tilde{r}_k} v_x \frac{ds}{s(1-s)}
    =& \GA{c-a} 
    \cdot \prod_{l=1}^m \GA{1-b_l} 
    \cdot \GA{\sum_{l=1}^{k-1}b_l -c}
    \cdot \GA{1-\sum_{l=1}^{k-1}b_l +c} \\
    & \nonumber
    \cdot \sum_{n_1 ,\ldots ,n_m=0}^{\infty}
    \frac{1}{\Ga_{n_1,\ldots ,n_m}^{(k)}(a,b,c)} 
    \cdot \prod_{l=1} ^{k-1} \left( \frac{x_k}{x_l} \right)^{n_l}
    \cdot x_k^{n_k}
    \cdot \prod_{l=k+1}^{m} \left( \frac{x_l}{x_k} \right)^{n_l} .
  \end{align}
\end{Lem}
\begin{proof}
  Note that if $s$ belongs to $C_0 \cup [\vep ,1-\vep] \cup C_1$, 
  it satisfies $\vep < |s| < 1+\vep$. 
  Since 
  \begin{align*}
    & \left| \frac{x_k}{x_l}\frac{1}{s} \right| 
    <\xi^{k-l} \cdot \frac{1}{\vep}<1 \quad (1 \leq l \leq k-1), \\
    & \left| x_k \frac{1}{s} \right| 
    <\xi^k \cdot \frac{1}{\vep}<1 , \\
    & \left| \frac{x_l}{x_k} s \right| 
    <\xi^{l-k} \cdot (1+\vep)<1 \quad (k+1 \leq l \leq m) ,
  \end{align*}
  the following power series expansions are uniformly and absolutely convergent 
  on $C_0 \cup [\vep ,1-\vep] \cup C_1$: 
  \begin{align*}
    & \left( 1-\frac{x_k}{x_l} \frac{1}{s} \right)^{\al_l} 
    =\sum_{n_l=0}^{\infty} \frac{(-\al_l,n_l)}{n_l !}
    \left( \frac{x_k}{x_l} \frac{1}{s} \right)^{n_l}
    \quad (1 \leq l \leq k-1), \\
    & \left( 1-x_k \frac{1}{s} \right) ^{\al_{m+1}-1}  
    =\sum_{n_k=0}^{\infty} \frac{(1-\al_{m+1},n_k)}{n_k !}
    \left( x_k \frac{1}{s} \right)^{n_k}, \\
    & \left( 1-\frac{x_l}{x_k} s \right)^{\al_l} 
    =\sum_{n_l=0}^{\infty} \frac{(-\al_l,n_l)}{n_l !}
    \left( \frac{x_l}{x_k} s \right)^{n_l}  \quad (k+1 \leq l \leq m).
  \end{align*}
  We replace the power functions on  
  the left-hand side of (\ref{homology-lem-1}) 
  by these expansions, and  
  exchange the sum and the integral. 
  Then, the coefficient of 
  $\DS \prod_{l=1}^{k-1} \left( \frac{x_k}{x_l} \right)^{n_l}
  \cdot x_k^{n_k}
  \cdot \prod_{l=k+1}^{m} \left( \frac{x_l}{x_k} \right)^{n_l}$ 
  is 
  \begin{align}
    \label{homology-lem-2}
    &\frac{(1-\al_{m+1},n_k)}{n_k !}
    \prod_{l\neq k} \frac{(-\al_l,n_l)}{n_l !} \\
    & \nonumber 
    \cdot \int_{\tilde{r}_k} 
    s^{\sum_{l=1}^{k-1}\al_l +\al_{m+1}+\al_{m+2}-1-\sum_{l=1}^{k}n_l+\sum_{l=k+1}^{m}n_l} 
    \cdot (1-s)^{\al_k+1} \frac{ds}{s(1-s)} .
  \end{align}
  By the construction of $\tilde{r}_k$, 
  the twisted cycle $\tilde{r}_k$ of this integral can be  
  identified with the usual regularization of the open interval $(0,1)$ 
  loaded with the multivalued function 
  $$
  s^{\sum_{l=1}^{k-1}\al_l +\al_{m+1}+\al_{m+2}-1-\sum_{l=1}^{k}n_l+\sum_{l=k+1}^{m}n_l} 
  \cdot (1-s)^{\al_k+1}
  $$
  on $\C -\{ 0,1 \}$. 
  Hence the integral in (\ref{homology-lem-2}) is equal to
  \begin{align*}
    \frac{\GA{\sum_{l\leq k-1}\al_l +\al_{m+1}+\al_{m+2}-1-\sum_{l\leq k}n_l+\sum_{l\geq k+1}n_l}  
      \GA{\al_k+1}}
    {\GA{\sum_{l\leq k}\al_l +\al_{m+1}+\al_{m+2}-\sum_{l\leq k}n_l+\sum_{l\geq k+1}n_l}} .
  \end{align*}
  By (\ref{alpha}) and $(a,n)=\GA{a+n}/ \GA{a}$, 
  (\ref{homology-lem-2}) is equal to 
  \begin{align*}
    & \frac{\GA{1-c+a+n_k}}{\GA{1-c+a}}
    \prod_{l\neq k} \frac{\GA{b_l+n_l}}{\GA{b_l}} \\
    & \cdot 
    \frac{\GA{-\sum_{l\leq k-1}b_l +c-1-\sum_{l\leq k}n_l+\sum_{l\geq k+1}n_l}  
      \GA{1-b_k}}
    {\GA{-\sum_{l\leq k}b_l +c-\sum_{l\leq k}n_l+\sum_{l\geq k+1}n_l}}
    \prod_{l=1}^m \frac{1}{\GA{1+n_l}} .
  \end{align*}
  By using $\GA{z} \GA{1-z} = \pi / \sin (\pi z)$, 
  we obtain, 
  for example, 
  \begin{align*}
    \frac{\GA{b_l+n_l}}{\GA{b_l}}
    &=\frac{1}{\GA{b_l}}\cdot \frac{\pi}{\GA{1-b_l-n_l} \sin \pi (b_l+n_l)} \\
    &=(-1)^{n_l} \frac{1}{\GA{1-b_l-n_l}}\frac{\pi}{\GA{b_l} \sin \pi b_l}
    =(-1)^{n_l} \frac{\GA{1-b_l}}{\GA{1-b_l-n_l}} .
  \end{align*}  
  In an analogous way, other Gamma functions with $n_l$'s in the numerator 
  can be moved to the denominator. 
  Thus, we obtain the lemma. 
\end{proof}

We will construct a twisted cycle standing for the series solution 
$f^{(k)}(a,b,c;x)$ by the bijection 
$$
\iota :S_x \to T_x ; \quad s \mapsto t=\frac{x_k}{s} .
$$
Let $r_k$ be the twisted cycle defined as $r_k:=\iota_{*}(\tilde{r}_k)$, 
which gives an element in $H_1 (T_x ,u_x)$. 
\begin{Th}\label{homology-th}
    \begin{align*}
    \int_{r_k} u_x \vph_0 
    =&  
    \GA{c-a} 
    \cdot \prod_{l=1}^m \GA{1-b_l} 
    \cdot \GA{\sum_{l=1}^{k-1}b_l -c}
    \cdot \GA{1-\sum_{l=1}^{k-1}b_l +c} \\
    & \cdot e^{\pi \sqrt{-1}(\sum_{l=1}^{k-1}b_l -c+a)}
    \cdot f^{(k)}(a,b,c;x). 
  \end{align*}
\end{Th}
\begin{proof}
  Note that $\arg (x_k)=-\pi$. 
  We have 
  \begin{align*}
    &u_x (\iota (s)) \cdot \iota^{*} \vph_0 \\
    &= \left( \frac{x_k}{s} \right)^{\al_0} 
    \cdot \left( \frac{x_k}{s}-1 \right)^{\al_{m+1}}
    \cdot \prod_{l=1}^{k-1} \left( \frac{x_k}{s}-x_l \right)^{\al_l}
    \cdot \left( \frac{x_k}{s}-x_k  \right) ^{\al_k} 
    \cdot \prod_{l=k+1}^{m} \left( \frac{x_k}{s}-x_l \right)^{\al_l} \\
    & \quad \cdot \frac{-x_k ds}{s^2 (\frac{x_k}{s}-1)} \\
    &=- \prod_{l=1}^{k-1} x_l^{\al_l} 
    \cdot x_k^{\al_0 +\sum_{l=k}^{m}\al_l +1} 
    \cdot s^{-\al_0 -\sum_{l=k}^{m}\al_l-1} 
    \cdot \left( \frac{x_k}{s}-1 \right)^{\al_{m+1}-1} \\
    &\quad \cdot \prod_{l=1}^{k-1} \left( \frac{x_k}{x_l s}-1 \right)^{\al_l}
    \cdot \left( 1-s  \right) ^{\al_k+1} 
    \cdot \prod_{l=k+1}^{m} \left( 1-\frac{x_l s}{x_k} \right)^{\al_l} 
    \cdot \frac{ds}{s(1-s)}\\
    &=e^{-\pi \sqrt{-1} (\sum_{l=1}^{k-1} \al_m +\al_{m+1})}
    \cdot \prod_{l=1}^{k-1} x_l^{\al_l} 
    \cdot x_k^{\al_0 +\sum_{l=k}^{m}\al_l+1}
    \cdot v_x (s)\frac{ds}{s(1-s)} .
  \end{align*}
  Here, we use $-\al_0-\sum_{l=k}^{m} \al_m= \sum_{l=1}^{k-1}\al_l +\al_{m+1}+\al_{m+2}$. 
  By Lemma \ref{homology-lem} and the relations 
  \begin{align*}
    &\al_l =-b_l \quad (1 \leq l \leq k-1),\\
    &\al_0 +\sum_{l=k}^{m}\al_l+1 
    =-\sum_{l=1}^{k-1}\al_l -\al_{m+1}-\al_{m+2}+1=\sum_{l=1}^{k-1}b_l -c+1, \\
    &\sum_{l=1}^{k-1} \al_m +\al_{m+1}
    =-\sum_{l=1}^{k-1} b_l +c-a ,
  \end{align*}
  we obtain the identity of the theorem. 
\end{proof}

By replacing the cycle $(1,\infty)$ in Section \ref{sec-contiguity} 
with $r_k$, 
we can obtain the contiguity relations of $f^{(k)}$. 
We put
\begin{align*}
  F^{(k)} (a,b,c;x)
  :=\TP{ \left( f^{(k)}(a,b,c;x), \frac{x_1-1}{-b_1}\frac{\pa}{\pa x_1}f^{(k)}(a,b,c;x),\
      \ldots ,\frac{x_m-1}{-b_m}\frac{\pa}{\pa x_m}f^{(k)}(a,b,c;x) \right)} .
\end{align*}
By Theorem \ref{homology-th} and \cite{M-FD}, we have 
\begin{align*}
  &\tilde{F}^{(k)}(a,b,c;x) 
  :=\TP{ \left( \int_{r_k} u_x \vph_0 ,
      \ldots ,\int_{r_k} u_x \vph_m \right)} \\
  &=\GA{c-a} 
  \cdot \prod_{l=1}^m \GA{1-b_l} 
  \cdot \GA{\sum_{l=1}^{k-1}b_l -c}
  \cdot \GA{1-\sum_{l=1}^{k-1}b_l +c} 
  \cdot e^{\pi \sqrt{-1}(\sum_{l=1}^{k-1}b_l -c+a)} 
  \cdot F^{(k)} (a,b,c;x) .
\end{align*}
It is clear that Lemmas \ref{contiguity-lem} and \ref{contiguity-cor} 
hold even if $\tilde{F}$ is replaced by $\tilde{F}^{(k)}$. 
Therefore, we obtain the following corollary. 
\begin{Cor}\label{contiguity-F^k}
    \begin{align*}
    F^{(k)}(a-1,b,c;x)&=D_a^{(k)} (a,b,c;x)F^{(k)}(a,b,c;x), \\
    F^{(k)}(a,b,c-1;x)&=D_c^{(k)} (a,b,c;x)F^{(k)}(a,b,c;x), \\
    F^{(k)}(a,b-e_l,c;x)&=D_l^{(k)} (a,b,c;x)F^{(k)}(a,b,c;x) \quad  (1\leq l \leq m), 
  \end{align*}
  where 
  \begin{align*}
    D_a^{(k)} (a,b,c;x)&:=\frac{1}{a-c} \cdot Q_{m+1}(a,b,c;x) \cdot C(a,b,c)^{-1}, \\
    D_c^{(k)} (a,b,c;x)&:=(c-a-1) \cdot Q_0(a+1,b,c;x) \cdot Q_{m+1}(a+1,b,c;x)^{-1}, \\
    D_l^{(k)} (a,b,c;x)&:=\frac{1}{1-b_l} \cdot Q_l(a+1,b,c+1;x) \cdot Q_0(a+1,b,c+1;x)^{-1} .
  \end{align*}
  In fact, $D^{(k)}_{\bullet}$ is independent of $k$. 
\end{Cor}

\section{Application---Normalizing constant for $2\times (m+1)$ contingency tables}
\label{sec-application}
Contiguity relations of $F_D$ and $f^{(k)}$ are 
applied to the numerical evaluation of 
the normalizing constant of 
the hypergeometric distribution of 
the $2\times (m+1)$ contingency tables 
with fixed marginal sums. 
In this section, we explain how our results are applied. 

We consider the $2 \times (m+1)$ contingency table
$$
u=\mat{cccc}{u_{10}&u_{11}&\cdots&u_{1m} \\ u_{20}&u_{21}&\cdots&u_{2m}} 
\in M_{2,m+1}(\Z_{\geq 0})
$$
with row sums $\be_1$ and $\be_2$ and columns sums $\ga_0,\ldots ,\ga_m$. 
We put $t:=\be_1+\be_2=\sum_{i=0}^m \ga_i$. 
We use the multi-index notation 
$$
p^u =\prod_{i=1}^2 \prod_{j=0}^m p_{ij}^{u_{ij}},\quad
u!=\prod_{i=1}^2 \prod_{j=0}^m u_{ij} ! ,
$$
where $p$ is the $2 \times (m+1)$ matrix variable. 
The polynomial 
$$
Z(\be,\ga ;p)=t! \sum_u \frac{p^u}{u!}
$$
is called the normalizing constant, 
where the sum is taken over all contingency tables $u$ with 
marginal sums $\be=(\be_1,\be_2)$ and 
$\ga=(\ga_0 ,\ldots ,\ga_m )$. 
It is a fundamental problem in statistics 
to evaluate $Z(\be, \ga; p)$ numerically, where 
$\be_i, \ga_j \in \Z_{\geq 0}$ and $p_{ij} \in \Q_{\geq 0}$. 

The normalizing constant $Z$ can be expressed 
by $F_D$ or $f^{(k)}$. 
To explain this, we will first define some notation. 
We put 
\begin{align*}
  &\CB_0 :=\{ (\be_1,\be_2, \ga_0 ,\ldots ,\ga_m ) \in (\Z_{>0})^{m+3} \mid 
  \be_1+\be_2=\sum_{i=0}^m \ga_i,\  \be_1 -\ga_0 \leq 0 \} , \\ 
  &\CB_k :=\{ (\be_1,\be_2 ,\ga_0 ,\ldots ,\ga_m ) \in (\Z_{>0})^{m+3} \mid 
  \be_1+\be_2=\sum_{i=0}^m \ga_i,\ 
  \be_1-\sum_{i=0}^{k-1}\ga_i >0 ,\  
  \be_1-\sum_{i=0}^{k}\ga_i \leq 0 \} ,
\end{align*}
where $1\leq k \leq m$. 
Then, 
$\{ (\be_1,\be_2, \ga_0 ,\ldots ,\ga_m ) \in (\Z_{>0})^{m+3} \mid 
\be_1+\be_2=\sum_{i=0}^m \ga_i  \}$ 
is the disjoint union of $\CB_0 ,\ldots ,\CB_m$. 
We also put 
\begin{align*}
  &\ell_1:=\mat{ccccc}{-1&1&0&\cdots&0 \\ 1&-1&0&\cdots&0} ,\quad 
  \ell_2:=\mat{cccccc}{-1&0&1&0&\cdots&0 \\ 1&0&-1&0&\cdots&0} ,\ldots ,\\ 
  &\ell_m:=\mat{ccccc}{-1&0&\cdots&0&1 \\ 1&0&\cdots&0&-1} ,\\
  &u_0:=\mat{ccccc}{\be_1&0&0&\cdots&0 \\ \ga_0-\be_1&\ga_1&\ga_2&\cdots&\ga_m}, \\
  &u_k:=\mat{ccccccc}{
    \ga_0&\cdots&\ga_{k-1}&\be_1-\sum_{i=0}^{k-1}\ga_i&0&\cdots&0 \\ 
    0&\cdots&0&\sum_{i=0}^{k}\ga_i -\be_1&\ga_{k+1}&\cdots&\ga_m
  } \quad
  (1\leq k \leq m), \\
  & x_i:=p^{\ell_i}=\frac{p_{1i}p_{20}}{p_{10}p_{2i}} .
\end{align*}
If $(\be_1,\be_2 ,\ga_0 ,\ldots ,\ga_m ) \in \CB_k$, then 
all of the entries of $u_k$ are non-negative integers, and hence 
it is one of the contingency tables with 
marginal sums $\be$ and $\ga$. 
By straightforward calculation, we can prove the following lemma. 
\begin{Lem}
  \begin{enumerate}
  \item If $(\be ,\ga) \in \CB_0$, then  
    \begin{align*}
      Z(\be,\ga;p)
      =\frac{t!}{u_0 !}
      \cdot p^{u_0} \cdot F_D(-\be_1 ,(-\ga_1 ,\ldots ,-\ga_m),\ga_0-\be_1+1 ;x_1,\ldots,x_m). 
    \end{align*}
  \item If $(\be ,\ga) \in \CB_k$ with $1\leq k \leq m$, then  
    \begin{align*}
      Z(\be,\ga;p)
      =t! \cdot p^{u_0} 
      \cdot f^{(k)}(-\be_1 ,(-\ga_1 ,\ldots ,-\ga_m),\ga_0-\be_1+1 ;x_1,\ldots,x_m) .
    \end{align*}
  \end{enumerate}
\end{Lem}

In \cite{OTT}, our contiguity relations are applied 
for the difference holonomic gradient method, 
which evaluates the numerical value of the column vector $F (a,b,c;x)$ or that of
$F^{(k)}(a,b,c;x)$, with $a, b_i \in \Z_{<0}$. 
For example, 
it follows from the below discussion of contiguity relations 
for integer parameters $a$, $b$, $c$ that 
we can easily evaluate the numerical value of $F(a,b,c;x)$ from 
that of $F(-1,b,c;x)$ by using the matrix $D_a$ in the contiguity relation. 
Note that 
$$
F_D (-1,b,c;x)=1-\sum_{i=1}^m \frac{b_i}{c} x_i .
$$
For details of the difference holonomic gradient method, 
see \cite{OTT}. 

We now consider the case in which the parameters are integers. 
Since $\be_1$, $\be_2$, $\ga_0$, $\ldots$, $\ga_m$ are integers, 
the parameters $(a,b,c)=(-\be_1 ,(-\ga_1 ,\ldots ,-\ga_m),\ga_0-\be_1+1)$ 
do not satisfy the condition (\ref{generic}). 
For the above application, 
we need to give the contiguity relations that are valid even when 
the parameters are integers.  

\begin{Prop}\label{contiguity-integer}
  \begin{enumerate}[(1)]
  \item If $(\be,\ga) \in \CB_0$, then the relation 
    \begin{align*}
      F(a-1,b,c;x)=\frac{a-1}{c-a} \cdot P_{m+1}(a,b,c;x)\cdot F(a,b,c;x)
    \end{align*}
    holds when the generic parameter vector is 
    specialized to an integral point 
    $(a,b,c) \to (-\be_1 ,(-\ga_1 ,\ldots ,-\ga_m),\ga_0-\be_1+1)$. 
  \item When $(\be,\ga) \in \CB_k$ with $1 \leq k \leq m$, 
    we consider the relation 
    \begin{align*}
      F^{(k)}(a-1,b,c-1;x)=-P_0(a,b,c;x) \cdot F^{(k)}(a,b,c;x).
    \end{align*}
    If $(\be_1+1 ,\be_2 ,\ga_0 ,\ldots ,\ga_m) \in \CB_k$, then 
    this relation 
    holds when the generic parameter vector is 
    specialized to an integral point 
    $(a,b,c) \to (-\be_1 ,(-\ga_1 ,\ldots ,-\ga_m),\ga_0-\be_1+1)$. 
  \end{enumerate}
\end{Prop}
We put 
\begin{align*}
  &\Ga_{n_1 ,\ldots ,n_m}^{(0)}(a,b,c) \\
  &:=\GA{1-a-\sum_{l=1}^m n_l} \cdot \GA{c+\sum_{l=1}^m n_l}
  \cdot \prod_{l=1}^m \GA{1-b_l-n_l} \cdot \prod_{l=1}^m \GA{1+n_l} .
\end{align*}
To prove this proposition, 
we will use the following lemma. 
\begin{Lem}
  \label{integer-convergent}
  \begin{enumerate}[(1)]
  \item Let $\tilde{a} ,\tilde{b}_1 ,\ldots ,\tilde{b}_m ,\tilde{c}$ be integers, 
    and assume $\tilde{c} >0$. 
    Then, 
    there exists $\tilde{x} \in \C^m$ such that 
    the power series 
    \begin{align*}
      \sum_{n_1,\ldots,n_m=0} ^{\infty } 
      \frac{1}{\Ga_{n_1 ,\ldots ,n_m}^{(0)}(a,b,c)} \prod_{l=1}^m x_l^{n_l}
    \end{align*}
    as a function in $(a,b,c;x)$ is holomorphic  
    on a small neighborhood of $(\tilde{a},\tilde{b},\tilde{c};\tilde{x})$. 
    In particular, if $x\in \C^m$ belongs to a small neighborhood of $\tilde{x}$, 
    then this series has a limit as $(a,b,c) \to (\tilde{a},\tilde{b},\tilde{c})$. 
  \item For $1 \leq k \leq m$, 
    let $\tilde{a} ,\tilde{b}_1 ,\ldots ,\tilde{b}_m ,\tilde{c}$ be integers that satisfy 
    \begin{align*}
      -\sum_{l\leq k} \tilde{b}_l+\tilde{c}>0, \quad 
      2+\sum_{l\leq k-1} \tilde{b}_l-\tilde{c}>0 .
    \end{align*}
    Then, there exists $\tilde{x} \in \C^m$ such that 
    the Laurent series 
    \begin{align*}
      \sum_{n_1 ,\ldots ,n_m=0}^{\infty}
      \frac{1}{\Ga_{n_1,\ldots ,n_m}^{(k)}(a,b,c)} 
      \cdot \prod_{l\leq k-1} \left( \frac{x_k}{x_l} \right)^{n_l}
      \cdot x_k^{n_k}
      \cdot \prod_{l\geq k+1} \left( \frac{x_l}{x_k} \right)^{n_l}
    \end{align*}
    as a function in $(a,b,c;x)$ is holomorphic  
    on a small neighborhood of $(\tilde{a},\tilde{b},\tilde{c};\tilde{x})$. 
    In particular, if $x\in \C^m$ belongs to a small neighborhood of $\tilde{x}$, 
    then this series has a limit as $(a,b,c) \to (\tilde{a},\tilde{b},\tilde{c})$. 
  \end{enumerate}
  Further, we can differentiate these series term by term, and 
  the partial derivatives of them also have limits as 
  $(a,b,c) \to (\tilde{a},\tilde{b},\tilde{c})$. 
\end{Lem}
We can show this lemma in a way that is analogous to
that used for \cite[Lemma 1]{OT}; 
see also \cite[pp.~18--21]{OT-kaken}. 
Although, in \cite{OT}, the parameter vector $(a,b,c)$ belongs to 
a neighborhood of a generic point, 
an analogous estimation of Gamma functions can be done in our case. 
\begin{proof}[Sketch of Proof]
  Let $0 \leq j \leq m$. 
  First, we can show that 
  there exist $C, \rho_1 ,\ldots ,\rho_m >0$ such that 
  the inequality 
  \begin{align*}
    \left| \frac{1}{\Ga_{n_1 ,\ldots ,n_m}^{(j)}(a,b,c)} \right|
    \leq C \rho_1^{n_1}\cdots \rho_m^{n_m}
  \end{align*}
  holds on a small neighborhood of $(\tilde{a},\tilde{b},\tilde{c})$. 
  Next, we put 
  $$
  \rho:=\max \{ \rho_1 ,\ldots ,\rho_m ,2  \} ,\quad 
  \tilde{x}_i :=\frac{1}{\rho ^{2i}} ,
  $$
  and $\tilde{x}:=(\tilde{x}_1, \ldots ,\tilde{x}_m)$. 
  We can show that 
  there exists $0 <\eta <1$ such that 
  the series in the lemma has the form 
  \begin{align*}
    \sum_{n_1 ,\ldots ,n_m =0}^{\infty} 
    \left( \prod_{l=1}^m \eta^{n_l} \right)
  \end{align*}
  for a majorant on a small neighborhood of 
  $(\tilde{a},\tilde{b},\tilde{c};\tilde{x})$. 
  Therefore, the series is uniformly and absolutely convergent, and it
  defines a holomorphic function. 
\end{proof}

\begin{proof}[Proof of Proposition \ref{contiguity-integer}]
  If $(a,b,c) = (-\be_1 ,(-\ga_1 ,\ldots ,-\ga_m),\ga_0-\be_1+1)$, 
  then
  the $\al_i$'s 
  are expressed as follows: 
  \begin{align*}
    & \al_0 =\be_1 -\sum_{l=0}^m \ga_l -1= -\be_2-1 ,\quad 
    \al_k =\ga_k \ (1 \leq k \leq m) ,\\
    & \al_{m+1} =\ga_0+1 ,\quad 
    \al_{m+2} =-\be_1 .
  \end{align*}
  Since these values and $1-\al_{m+2}=\be_1+1$ are not zero, 
  it follows from
  Fact \ref{intersection}, Proposition \ref{Q}, and 
  Remark \ref{detQ} that 
  both of the matrices 
  $Q_k(a,b,c;x)$ and 
  $C(a,b,c)$ are 
  well-defined and invertible.   
  \begin{enumerate}[(1)]
  \item The definition of $F_D(a,b,c;x)$ can be expressed by 
    the Gamma function: 
    \begin{align*}
      F_D(a,b,c;x)
      =\GA{1-a} \cdot \GA{c}
      \cdot \prod_{l=1}^m \GA{1-b_l} 
      \cdot \sum_{n_1,\ldots,n_m=0} ^{\infty } 
      \frac{1}{\Ga_{n_1 ,\ldots ,n_m}^{(0)}(a,b,c)} \prod_{l=1}^m x_l^{n_l} .
    \end{align*}
    By $(\be ,\ga) \in \CB_0$, 
    we have $c=\ga_0-\be_1+1>0$. 
    Then we can apply Lemma \ref{integer-convergent} (1) 
    to $F(a,b,c;x)$. 
    Note that $a-1=-\be_1 -1 \neq 0$, and $c-a=\ga_0+1 \neq 0$. 
  \item Let $\sigma$ be $0$ or $1$. 
    $(\be_1+\sigma ,\be_2 ,\ga_0 ,\ldots ,\ga_m) \in \CB_k$
    implies 
    \begin{align*}
      &-\sum_{l=1}^{k} b_l+(c-\sigma)
      =-(\be_1 +\sigma )+\sum_{l=0}^k \ga_l +1 >0, \\ 
      &2+\sum_{l=1}^{k-1} b_l-(c-\sigma)
      =(\be_1 +\sigma) -\sum_{l=0}^{k-1} \ga_l +1 >0.
    \end{align*}
    Then we can take the limit of $F^{(k)} (a,b,c;x)$ as 
    $(a,b,c) \to (-\be_1 ,(-\ga_1 ,\ldots ,-\ga_m),\ga_0-\be_1+1)$ 
    by Lemma \ref{integer-convergent} (2). 
  \end{enumerate}
  By the identity theorem for holomorphic functions, 
  it is sufficient to prove 
  the proposition on a small neighborhood of 
  some $x \in \C^m$. 
  Therefore, the proof is completed. 
\end{proof}

\section*{Acknowledgments}
The author thanks Professor Nobuki Takayama 
for posing this problem and for his constant encouragement.
This work was supported by JSPS KAKENHI Grant Numbers 
25220001, 26-1252.

\end{document}